\documentclass[12pt]{article}
\usepackage[margin=1in]{geometry}
\usepackage{tikz}
\usepackage{amsmath,amssymb,amsthm}
\usepackage[hidelinks]{hyperref}

\newtheorem{theorem}{Theorem}[section]
\newtheorem{proposition}{Proposition}[section]
\newtheorem{corollary}{Corollary}[section]
\newtheorem{lemma}{Lemma}[section]
\newtheorem{claim}{Claim}[section]
\theoremstyle{definition}
\newtheorem{observation}{Observation}[section]
\newtheorem{question}{Question}[section]
\newtheorem{definition}{Definition}[section]

\DeclareMathOperator{\lng}{long}
\DeclareMathOperator{\cy}{cy}

\title{Spectral faux trees}
\author{Steve Butler\footnote{Iowa State University, Ames, IA 50011, USA, \texttt{\{butler,joeljef}@iastate.edu}\and 
Elena D'Avanzo\footnote{Carleton College, Northfield, MN 55057, USA, \texttt{davanzoe@carleton.edu}} \and 
Rachel Heikkinen\footnote{Augustana College, Rock Island, IL, 61201, USA, \texttt{rachelheikkinen19@augustana.edu}} \and 
Joel Jeffries\footnotemark[1] \and 
Alyssa Kruczek\footnote{Susquehanna University, Selinsgrove, PA 17870, USA, \texttt{a.kruczek0413@gmail.com}} \and
Harper Niergarth\footnote{University of Minnesota--Twin Cities, Minneapolis, MN 55455, USA, \texttt{nierg001@umn.edu}}}
\date{\empty}

\begin{document}

\maketitle

\begin{abstract}
A spectral faux tree with respect to a given matrix is a graph which is not a tree but is cospectral with a tree for the given matrix. We consider the existence of spectral faux trees for several matrices, with emphasis on constructions.

For the Laplacian matrix, there are no spectral faux trees. For the adjacency matrix, almost all trees are cospectral with a faux tree. For the signless Laplacian matrix, spectral faux trees can only exist when the number of vertices is of the form $n=4k$. For the normalized adjacency, spectral faux trees exist when the number of vertices $n\ge 4$, and we give an explicit construction for a family whose size grows exponentially with $k$ for $n=\alpha k+1$ where $\alpha$ is fixed.
\end{abstract}

\section{Introduction}\label{sec:introduction}
Kac \cite{kac} posed the famous question of whether you can hear the shape of a drum. For spectral graph theory, this question is often phrased in terms of whether a graph is uniquely determined by its spectrum for a given matrix (see \cite{vandam}). The problem in general is wide open, though numerous special graphs have been shown to be determined by their spectrum, and at the same time many special constructions of families of graphs with the same spectrum are known. One of the strongest results in this direction is due to Schwenk \cite{schwenk} who was able to show that almost all trees share a spectrum with a different tree with respect to the adjacency matrix.

We will consider the situation of trees being cospectral with non-trees for several different matrices. In particular, we are interested in what we term spectral faux trees, which are defined below. Note that all graphs considered in this paper, unless stated otherwise, are simple graphs.

\begin{definition}
A \emph{spectral faux tree with respect to a matrix} is a graph which is not a tree but has the same spectrum (eigenvalues, including multiplicity) as some tree with respect to the matrix.
\end{definition}

We will be considering several of the more commonly studied matrices, the adjacency ($A$, which is a $0$-$1$ matrix with rows/columns indexed by vertices and whose entries indicate whether corresponding vertices are adjacent), the Laplacian ($L=D-A$, where $D$ is the diagonal matrix of degrees), the signless Laplacian ($Q=D+A$), and the normalized adjacency matrix ($\mathcal{A}=D^{-1/2}AD^{-1/2}$). Each of these matrices has different strengths (i.e.\ ability to discern graph structures) as well as different weaknesses (i.e.\ examples of differing graphs with the same spectrum) (see \cite{haemers,butler}). Therefore, it is not surprising that when it comes to spectral faux trees, there is a significant difference in the behavior of the matrices. We summarize our results here.
\begin{itemize}
\item[$A$:] For the adjacency matrix, almost all trees are cospectral with a non-tree, and in some sense there are many more spectral faux trees then there are trees. (Section~\ref{sec:ALS})
\item[$L$:] For the Laplacian matrix, there do not exist any spectral faux trees. (Section~\ref{sec:ALS})
\item[$Q$:] For the signless Laplacian matrix, spectral faux trees can only exist if the number of vertices is of the form $n=4k$. Moreover, for each such $n$ there do exist spectral faux trees, and the number of spectral faux trees grows exponentially as a function of $k$. (Section~\ref{sec:ALS})
\item[$\mathcal{A}$:] For the normalized adjacency matrix (equivalently the normalized Laplacian, or probability transition matrix), spectral faux trees exist for any number of vertices $n\ge4$. Moreover, we construct an explicit family of faux trees for $n=\alpha k+1$ vertices, for fixed $\alpha$, which grows exponentially as a function of $k$. (Section \ref{sec:branching})
\end{itemize}

To get a sense of the growth of spectral faux trees, in Table~\ref{tab:AA} we give the number of trees, spectral faux trees, and trees which are cospectral with a non-tree for the matrices $A$ and $\mathcal{A}$ for $4\le n\le14$. In Table~\ref{tab:Q}, we give similar data for the matrix $Q$ for $n=4,8,12,16$. This data reinforces the above comment that there is significant difference in the behavior of the matrices.

\begin{table}[!htb]
\centering
\begin{tabular}{|l||r|r|r|r|r|r|r|r|r|r|r|r|r|} \hline
\hfill$n=$&4&5&6&7&8&9&10&11&12&13&14\\ \hline
\#Trees&2&3&6&11&23&47&106&235&551&1301&3159\\ \hline \hline
\#Faux-trees ($A$)&0&1&1&6&5&30&67&244&480&1843&4226 \\ \hline
\#Trees cospectral&0&1&1&6&5&24&50&140&291&838&1882 \\[-3pt] 
\phantom{\#}w/ non-tree ($A$)&&&&&&&&&&&\\ \hline \hline
\#Faux-trees ($\mathcal{A}$)&1&1&2&4&3&8&11&27&9&119&20 \\ \hline
\#Trees cospectral&1&1&1&2&1&4&7&11&5&59&10 \\[-3pt] 
\phantom{\#}w/ non-tree ($\mathcal{A}$)&&&&&&&&&&&\\ \hline
\end{tabular}
\caption{Data related to spectral faux trees for $A$ and $\mathcal{A}$ on $n$ vertices.}
\label{tab:AA}
\end{table}

\begin{table}[!htb]
\centering
\begin{tabular}{|l||r|r|r|r|}\hline
\hfill$n=$&$4$&$8$&$12$&$16$ \\ \hline
\#Trees & 2& 23& 551& 19320\\ \hline\hline
\#Faux trees&1 &2 & 9&$48$\\ \hline
\#Trees cospectral &1 &2 & 9&$34$\\[-3pt]
 \phantom{\#}w/ non-tree&&&&\\\hline

\end{tabular}
\caption{Data related to spectral faux trees for $Q$ on $n$ vertices.}
\label{tab:Q}
\end{table}

We will proceed as follows. In Section~\ref{sec:ALS} we will discuss spectral faux trees for the adjacency, Laplacian, and signless Laplacian matrices. For these matrices, the results primarily follow from existing known techniques. We then turn our attention to the normalized adjacency in Section~\ref{sec:branching} and show how ``ornamented'' binary trees form a large family of graphs which are cospectral, where one is a tree and one is a non-tree. Finally, we give some concluding comments in Section~\ref{sec:conclusion}.

\section{Spectral faux trees for the adjacency, Laplacian, and signless Laplacian}\label{sec:ALS}

\subsection{Adjacency matrix}
For the adjacency matrix, we can modify the work of Schwenk \cite{schwenk} who found a construction showing that almost all trees are cospectral with another tree. In particular, we will need the following two results from Schwenk related to coalescing (combining two graphs together by identifying a common vertex and merging the graphs at that vertex).

\begin{theorem}[{Schwenk \cite{schwenk}}]\label{thm:schwenk1}
Given a rooted tree $T$ on $k$ vertices, then as $n\to\infty$, with high probability, given a tree $S$ on $n$ vertices, $T$ exists as a ``limb'' of $S$ (in other words, $S$ can be formed by coalescing $T$ at its root onto a smaller tree).
\end{theorem}

\begin{theorem}[{Schwenk \cite{schwenk}}]\label{thm:schwenk2}
Let $G_1$ and $G_2$ be rooted graphs with roots $v_1$ and $v_2$, respectively. If $G_1$ and $G_2$ are cospectral and $G_1-v$ and $G_2-v$ are cospectral, then for any rooted graph $H$ with root $u$, coalescing $G_1$ onto $H$ by identifying $u$ and $v_1$ will be cospectral with the graph resulting from coalescing $G_2$ onto $H$ by identifying $u$ and $v_2$.
\end{theorem}

In brief, the idea of Schwenk was if two small rooted graphs satisfy Theorem~\ref{thm:schwenk2}, then one can swap $G_1$ and $G_2$ and still be cospectral. Then for two such rooted trees, Theorem~\ref{thm:schwenk1} shows that almost all trees contain one of those rooted trees as a limb which can be swapped with the other to form a new cospectral tree. We are now ready to prove the following.

\begin{theorem}
Almost every tree is cospectral with a non-tree (a spectral faux tree) with respect to the adjacency matrix.
\end{theorem}
\begin{proof}
Using Theorems~\ref{thm:schwenk1} and \ref{thm:schwenk2}, it suffices to find a rooted tree which is cospectral with a rooted non-tree, and where removal of the root leaves two cospectral graphs. Such a pair is shown in Figure~\ref{fig:A_cospec_pair}. 
\end{proof}

\begin{figure}[htp!]
\begin{center}
\begin{tikzpicture}[scale=0.8]

    \pgfmathsetmacro{\r}{1};

    \node[thick,inner sep=2pt,thick,shape=circle,draw=black] (A) at (0,0) {};
    \node[thick,inner sep=3pt,thick,shape=circle,draw=black, fill=white!70!red] (B) at (60+30:\r) {};
    \node[thick,inner sep=2pt,thick,shape=circle,draw=black] (C) at (120+30:\r) {};
    \node[thick,inner sep=2pt,thick,shape=circle,draw=black] (D) at (180+30:\r) {};
    \node[thick,inner sep=2pt,thick,shape=circle,draw=black] (E) at (240+30:\r) {};
    \node[thick,inner sep=2pt,thick,shape=circle,draw=black] (F) at (300+30:\r) {};
    \node[thick,inner sep=2pt,thick,shape=circle,draw=black] (G) at (0+30:\r) {};
    
    \path [thick,-] (B) edge (C);
    \path [thick,-] (E) edge (F);
    \path [thick,-] (F) edge (G);
    \path [thick,-] (G) edge (B);
    \path [thick,-] (G) edge (A);
    \path [thick,-] (D) edge (A);
    
\end{tikzpicture}
\hspace{3em} 
\begin{tikzpicture}[scale=0.8]

    \pgfmathsetmacro{\r}{1};

    \node[thick,inner sep=2pt,thick,shape=circle,draw=black] (A) at (0,0) {};
    \node[thick,inner sep=3pt,thick,shape=circle,draw=black, fill=white!70!red] (B) at (60+30:\r) {};
    \node[thick,inner sep=2pt,thick,shape=circle,draw=black] (C) at (120+30:\r) {};
    \node[thick,inner sep=2pt,thick,shape=circle,draw=black] (D) at (180+30:\r) {};
    \node[thick,inner sep=2pt,thick,shape=circle,draw=black] (E) at (240+30:\r) {};
    \node[thick,inner sep=2pt,thick,shape=circle,draw=black] (F) at (300+30:\r) {};
    \node[thick,inner sep=2pt,thick,shape=circle,draw=black] (G) at (0+30:\r) {};

    \path [thick,-](B) edge (C);
    \path [thick,-](C) edge (D);
    \path [thick,-](D) edge (E);
    \path [thick,-] (E) edge (F);
    \path [thick,-] (F) edge (G);
    \path [thick,-] (G) edge (B);
    
\end{tikzpicture}
\end{center}
\caption{A special rooted $A$-cospectral tree and faux tree pair.}
\label{fig:A_cospec_pair}
\end{figure}
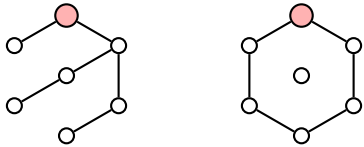

Examination of the work of Schwenk \cite{schwenk} actually indicates that for Theorem~\ref{thm:schwenk1}, not only is there likely one limb to swap, there are many limbs available. In particular, this indicates that every tree is cospectral with many non-trees. Then as $n$ gets large, the number of spectral faux trees with respect to the adjacency matrix should grow larger than the number of trees. This is reflected in Table~\ref{tab:AA}.

\subsection{Laplacian matrix}
For the Laplacian matrix, we will rely on the following two well-known facts about the spectrum of the Laplacian matrix (see \cite{butler}).

\begin{proposition}\label{fact:L1}
The sum of the eigenvalues of $L$ is twice the number of edges of the graph.
\end{proposition}

\begin{proposition}\label{fact:L2}
The multiplicity of $0$ as an eigenvalue of $L$ is the number of connected components of the graph.
\end{proposition}

\begin{theorem}
There are no $L$-spectral faux trees.
\end{theorem}

\begin{proof}
Using Propositions~\ref{fact:L1} and \ref{fact:L2}, the spectrum of the Laplacian can determine the number of edges and if a graph is connected. Since one characterization of trees on $n$ vertices are being connected and having $n-1$ edges, the spectrum can determine if a graph is a tree. So, there are no spectral faux trees for $L$.
\end{proof}

\subsection{Signless Laplacian matrix}
The signless Laplacian is a bit more subtle, and so we will need a bigger collection of facts. Most of these can again be found in standard references (e.g.\ \cite{butler}).

\begin{proposition}\label{prop:Qedges}
The sum of the eigenvalues of $Q$ is twice the number of edges of the graph.
\end{proposition}

\begin{proposition}\label{prop:Qcomp}
The multiplicity of $0$ as an eigenvalue of $Q$ is the number of connected \emph{bipartite} components of the graph.
\end{proposition}

\begin{proposition}\label{prop:same}
If a graph $G$ is bipartite, then $L$ and $Q$ have the same eigenvalues.
\end{proposition}

\begin{theorem}[Matrix Tree Theorem]\label{thm:MTT}
The product of the nonzero eigenvalues of $L$ for a graph on $n$ vertices is the product of $n$ and the number of spanning trees of the graph.
\end{theorem}

\begin{theorem}[{Hassani Monfared and Mallik \cite{signless_matrix_tree}}]\label{thm:QMTT}
If $G$ is a connected odd unicyclic graph, then the product of the eigenvalues of $Q$ is $4$.
\end{theorem}

\begin{theorem}
If $G$ is a spectral faux tree with respect to $Q$, then it has $n=4k$ vertices.
\end{theorem}

\begin{proof}
Suppose $T$ is a tree and $G$ a non-tree on $n$ vertices which are cospectral with respect to $Q$ (so $G$ is a spectral faux tree). Let $\lambda_1\le\cdots\le\lambda_n$ be their common spectrum. Since $T$ is a tree with one component, by Proposition~\ref{prop:Qcomp} we have that $0$ is an eigenvalue with multiplicity $1$, and so $\lambda_1=0$ and $\lambda_2>0$. So, $G$ must also have exactly one bipartite component and $n-1$ edges (by Proposition~\ref{prop:Qedges}). 

Since there are more vertices than edges in $G$, some component of $G$ must have more vertices than edges, and since the component is connected, that means at least one component of $G$ is a tree. Moreover, since $G$ cannot have two bipartite components, there is exactly one tree. For the remaining components, each component must be odd unicyclic; unicyclic because if one component were not unicyclic we would have too many or too few edges, and odd unicyclic because we cannot have another bipartite component. We conclude that $G=T'\cup U_1\cup \cdots \cup U_\ell$ where $T'$ is a tree on less than $n$ vertices and the $U_i$ are odd unicyclic graphs with $\ell\ge1$.

We now compute $\lambda_2\cdots\lambda_n$. By Proposition~\ref{prop:same}, we have that the eigenvalues of a tree are the same as for the Laplacian. Now, applying Theorem~\ref{thm:MTT}, this product is the number of vertices of $T$ (since there is only one spanning tree), which is $n$. Since $\lambda_1=0$ will be an eigenvalue of $T'$, we have that the product of the non-tree eigenvalues will give the determinants of the respective components. So we have
\[
n = |T|=\lambda_2\cdots\lambda_n=|T'|\det(U_1)\cdots\det(U_\ell)=|T'|4^\ell,
\]
where in the last step we used Theorem~\ref{thm:QMTT} (and the fact that the determinant of a graph is the product of its eigenvalues). In particular, we have that $4$ divides $n$.
\end{proof}

The preceding result shows that if spectral faux trees exist for $Q$, then they must have $n=4k$ vertices. It does not, however, show that spectral faux trees exist for all such $n$. We now establish their existence.

\begin{theorem}
There exist exponentially many spectral faux trees with respect to $Q$ on $n=4k$ vertices.
\end{theorem}

\begin{proof}
First note that the tree and faux tree in Figure~\ref{fig:Qpair1} are cospectral with respect to $Q$, so the result holds for $k = 1$.
\begin{figure}[!htb]
\begin{center}
\begin{tikzpicture}[scale=0.8]

    \pgfmathsetmacro{\r}{1};

    \node[inner sep=2pt,thick,shape=circle,draw=black] (A) at (90:\r) {};
    \node[inner sep=2pt,thick,shape=circle,draw=black] (B) at (210:\r) {};
    \node[inner sep=2pt,thick,shape=circle,draw=black] (C) at (330:\r) {};
    \node[inner sep=2pt,thick,shape=circle,draw=black] (D) at (0,0) {};

    \path [thick,-] (A) edge node[left] {} (B);
    \path [thick,-](B) edge node[left] {} (C);
    \path [thick,-](A) edge node[left] {} (C);
    
\end{tikzpicture}
\hspace{3em}
\begin{tikzpicture}[scale=0.8]

    \pgfmathsetmacro{\r}{1};

    \node[inner sep=2pt,thick,shape=circle,draw=black] (A) at (0,0) {};
    \node[inner sep=2pt,thick,shape=circle,draw=black] (B) at (90:\r) {};
    \node[inner sep=2pt,thick,shape=circle,draw=black] (C) at (210:\r) {};
    \node[inner sep=2pt,thick,shape=circle,draw=black] (D) at (330:\r) {};

    \path [thick,-] (A) edge (B);
    \path [thick,-](A) edge (C);
    \path [thick,-](A) edge (D);
    
\end{tikzpicture}
\end{center}
\caption{A cospectral tree and faux tree pair with respect to $Q$.}
\label{fig:Qpair1}
\end{figure}
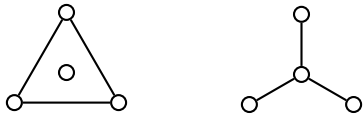

We can now build on the graphs in Figure~\ref{fig:Qpair1} by taking $4$ copies of any rooted graph $G$ and identifying each root with a distinct vertex in the cospectral graphs (see Figure~\ref{fig:Qpair2} for an example where we attach a star). By a result of the authors \cite[Corollary~2]{coalescing}, the resulting graphs remain cospectral with respect to $Q$. In particular, if $G$ is a tree on $k$ vertices, then the resulting graphs are a cospectral tree and faux tree pair on $4k$ vertices.

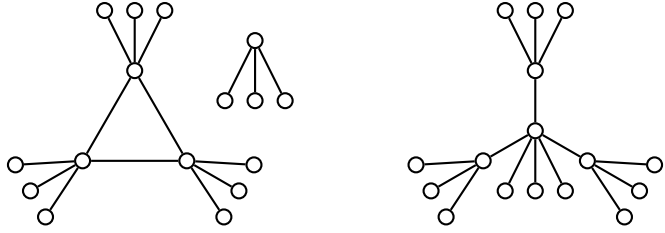
\begin{figure}
\begin{center}
\begin{tikzpicture}[scale=0.8]

    \pgfmathsetmacro{\r}{1};
    \pgfmathsetmacro{\rr}{2};
    \pgfmathsetmacro{\s}{0.5};

    \node[inner sep=2pt,thick,shape=circle,draw=black] (A) at (90:\r) {};
    \node[inner sep=2pt,thick,shape=circle,draw=black] (B) at (210:\r) {};
    \node[inner sep=2pt,thick,shape=circle,draw=black] (C) at (330:\r) {};
    \node[inner sep=2pt,thick,shape=circle,draw=black] (D) at (2,1.5) {};
    
    \node[inner sep=2pt,thick,shape=circle,draw=black] (A1) at (0,\rr) {};
    \node[inner sep=2pt,thick,shape=circle,draw=black] (A2) at (\s,\rr) {};
    \node[inner sep=2pt,thick,shape=circle,draw=black] (A3) at (-\s,\rr) {};
    
    \begin{scope}[rotate=120]
        \node[inner sep=2pt,thick,shape=circle,draw=black] (B1) at (0,\rr) {};
        \node[inner sep=2pt,thick,shape=circle,draw=black] (B2) at (\s,\rr) {};
        \node[inner sep=2pt,thick,shape=circle,draw=black] (B3) at (-\s,\rr) {};
    \end{scope}
    
    \begin{scope}[rotate=240]
        \node[inner sep=2pt,thick,shape=circle,draw=black] (C1) at (0,\rr) {};
        \node[inner sep=2pt,thick,shape=circle,draw=black] (C2) at (\s,\rr) {};
        \node[inner sep=2pt,thick,shape=circle,draw=black] (C3) at (-\s,\rr) {};
    \end{scope}
    
    \node[inner sep=2pt,thick,shape=circle,draw=black] (D1) at (2,0.5) {};
    \node[inner sep=2pt,thick,shape=circle,draw=black] (D2) at (2+\s,0.5) {};
    \node[inner sep=2pt,thick,shape=circle,draw=black] (D3) at (2-\s,0.5) {};

    \path [thick,-] (A) edge (B);
    \path [thick,-](B) edge (C);
    \path [thick,-](A) edge (C);
    
    \path [thick,-] (A) edge (A1);
    \path [thick,-] (A) edge (A2);
    \path [thick,-] (A) edge (A3);
    
    \path [thick,-] (B) edge (B1);
    \path [thick,-] (B) edge (B2);
    \path [thick,-] (B) edge (B3);
    
    \path [thick,-] (C) edge (C1);
    \path [thick,-] (C) edge (C2);
    \path [thick,thick,-] (C) edge (C3);
    
    \path [thick,-] (D) edge (D1);
    \path [thick,-] (D) edge (D2);
    \path [thick,-] (D) edge (D3);
    
\end{tikzpicture}
\hspace{3em}
\begin{tikzpicture}[scale=0.8]

    \pgfmathsetmacro{\r}{1};
    \pgfmathsetmacro{\rr}{2};
    \pgfmathsetmacro{\s}{0.5};

    \node[inner sep=2pt,thick,shape=circle,draw=black] (D) at (0,0) {};
    \node[inner sep=2pt,thick,shape=circle,draw=black] (A) at (90:\r) {};
    \node[inner sep=2pt,thick,shape=circle,draw=black] (B) at (210:\r) {};
    \node[inner sep=2pt,thick,shape=circle,draw=black] (C) at (330:\r) {};
    
    \node[inner sep=2pt,thick,shape=circle,draw=black] (A1) at (0,\rr) {};
    \node[inner sep=2pt,thick,shape=circle,draw=black] (A2) at (\s,\rr) {};
    \node[inner sep=2pt,thick,shape=circle,draw=black] (A3) at (-\s,\rr) {};
    
    \begin{scope}[rotate=120]
        \node[inner sep=2pt,thick,shape=circle,draw=black] (B1) at (0,\rr) {};
        \node[inner sep=2pt,thick,shape=circle,draw=black] (B2) at (\s,\rr) {};
        \node[inner sep=2pt,thick,shape=circle,draw=black] (B3) at (-\s,\rr) {};
    \end{scope}
    
    \begin{scope}[rotate=240]
        \node[inner sep=2pt,thick,shape=circle,draw=black] (C1) at (0,\rr) {};
        \node[inner sep=2pt,thick,shape=circle,draw=black] (C2) at (\s,\rr) {};
        \node[inner sep=2pt,thick,shape=circle,draw=black] (C3) at (-\s,\rr) {};
    \end{scope}
    
    \node[inner sep=2pt,thick,shape=circle,draw=black] (D1) at (0,-1) {};
    \node[inner sep=2pt,thick,shape=circle,draw=black] (D2) at (\s,-1) {};
    \node[inner sep=2pt,thick,shape=circle,draw=black] (D3) at (-\s,-1) {};

    \path [thick,-] (D) edge (A);
    \path [thick,-](D) edge (B);
    \path [thick,-](D) edge (C);
    
    \path [thick,-] (A) edge (A1);
    \path [thick,-] (A) edge (A2);
    \path [thick,-] (A) edge (A3);
    
    \path [thick,-] (B) edge (B1);
    \path [thick,-] (B) edge (B2);
    \path [thick,-] (B) edge (B3);
    
    \path [thick,-] (C) edge (C1);
    \path [thick,-] (C) edge (C2);
    \path [thick,-] (C) edge (C3);
    
    \path [thick,-] (D) edge (D1);
    \path [thick,-] (D) edge (D2);
    \path [thick,-] (D) edge (D3);
    
\end{tikzpicture}
\end{center}
\caption{The $Q$-cospectral tree and faux tree pair obtained by attaching a rooted star to each vertex of the graphs in Figure~\ref{fig:Qpair1}}
\label{fig:Qpair2}
\end{figure}
 
Since we can use an arbitrary rooted tree in this construction, and the number of trees on $k$ vertices grows exponentially, the number of spectral faux trees also grows exponentially as a function $k$.
\end{proof}

While we have shown that the number of trees grows exponentially, the arguments used produce trees with high amounts of symmetry, which is uncommon. Based on this and the limited data from Table~\ref{tab:Q}, the authors believe that spectral faux trees are rare for the signless Laplacian.

\section{Normalized adjacency and ornamented binary trees} \label{sec:branching}
We now consider the normalized adjacency matrix, $\mathcal{A}=D^{-1/2}AD^{-1/2}$ (we will only work with graphs without isolated vertices to avoid dealing with division by $0$). As mentioned in Section~\ref{sec:introduction}, $\mathcal{A}$ is closely related to the probability transition matrix ($P=D^{-1}A$) and the normalized Laplacian matrix ($\mathcal{L}=I-D^{-1/2}AD^{-1/2}$), so results of cospectrality for any one of these will also hold for the others. 

The main goal of this section will be the construction of large families of spectral faux trees. So, we will introduce a family of trees which are cospectral with non-trees. The proof of cospectrality will be done by establishing that the characteristic polynomials are equal (or more precisely, by showing processes to compute the characteristic polynomials that will result in equivalent expressions). Much of this work has similarities to earlier constructions of Butler and Heysse \cite{heysse}, but whereas that construction was circular, ours will involve branching.

\subsection{Characteristic polynomial of the normalized adjacency}
We first establish basic facts and properties of the characteristic polynomial for the normalized adjacency. We begin with the following definitions.

\begin{definition}
A \emph{cycle decomposition} $C$ is a subgraph without isolated vertices consisting of disjoint edges and cycles. We further have that $u(C)$ is the number of vertices not used in $C$ (vertices not in any cycle or edge); $\lng(C)$ is the number of cycles of length $\ge3$; and $\cy(C)$ is the number of cycles of $C$, where an edge is considered a cycle of length $2$.
\end{definition}

Note that in some contexts, it is useful to consider isolated vertices as part of the cycle decomposition, so it is a decomposition of the vertices. For our purposes, it is convenient to leave out these vertices. With this notation in place, we now have a combinatorial way to compute the characteristic polynomial of the normalized adjacency matrix.

\begin{proposition}\label{prop:allC}
For a graph $G$, let $\mathcal{C}$ be the set of all cycle decompositions. Then, the characteristic polynomial $p(x)$ of $G$ with respect to $\mathcal{A}$ can be written as,
\begin{equation}\label{eq:allC}
p(x)=\sum_{C\in\mathcal{C}}x^{u(C)}2^{\lng(C)}(-1)^{\cy(C)}\prod_{v\in C}\frac1{\deg(v)}.
\end{equation}
\end{proposition}

The proof of this is nearly identical to that for the normalized Laplacian in Butler and Heysse \cite[Proposition~1]{heysse}, and we refer the reader there for the details.

Looking ahead, our construction will involve complete bipartite graphs as blocks that have been glued together in some way. We want to use this information to simplify the computation of the characteristic polynomial as much as possible. This is the purpose of the following results.

\begin{proposition}\label{prop:Ktt for A}
For $t\ge2$, let $\overline{\mathcal{C}}$ denote the collection of cycle decompositions using all vertices of $K_{t,t}$. Then
\[
\sum_{C\in\overline{\mathcal{C}}}2^{\lng(C)}(-1)^{\cy(C)} = 0.
\]
\end{proposition}
\begin{proof}
This sum is equivalent to computing the determinant of the adjacency matrix of $K_{t,t}$ (see Brualdi and Ryser \cite{brualdi}; alternatively this sum is also the constant term of the characteristic polynomial of the adjacency). Since the adjacency matrix of $K_{t,t}$ is rank deficient, the determinant is $0$.
\end{proof}

\begin{definition}
Given a graph $G$, a subgraph $H$, and a cycle decomposition $C$ of $G$, let $\mathcal{I}_H(C)$ be the set of all edges or cycles of $C$ that are wholly contained inside of $H$. We call $\mathcal{I}_H(C)$ the \emph{internal cycles} of $H$ with respect to $C$.
\end{definition}

\begin{lemma}\label{lem:usefulC}
For a graph $G$, let $H_1,H_2,\ldots,H_\ell$ be \emph{edge disjoint} complete bipartite graphs, and let $\widehat{\mathcal{C}}$ be the set of all cycle decompositions so that for all $C\in\widehat{\mathcal{C}}$ and $H_i$, we have $\mathcal{I}_{H_i}(C)$ is either empty or consists of a single edge. Then, the characteristic polynomial $p(x)$ of $G$ with respect to $\mathcal{A}$ can be written as,
\[
p(x)=\sum_{C\in\widehat{\mathcal{C}}}x^{u(C)}2^{\lng(C)}(-1)^{\cy(C)}\prod_{v\in C}\frac1{\deg(v)}.
\]
\end{lemma}

Comparing this result to Proposition~\ref{prop:allC}, the key difference is that we can simplify our cycle decompositions to the point where, in the identified complete bipartite graphs, we have either no cycles or a single edge. Before starting the proof, we also note that the same proof technique will establish a similar result for the adjacency matrix (though we do not need that result here).

\begin{proof}[Proof of Lemma~\ref{lem:usefulC}]
First, consider a single complete bipartite graph $H_1$. We can partition the collection of all cycle decompositions of $G$ into equivalence classes so that $C_1$ and $C_2$ are two cycle decompositions in the same equivalence class $\mathcal{C}'$ if and only if $C_1\setminus\mathcal{I}_{H_1}(C_1)=C_2\setminus\mathcal{I}_{H_1}(C_2)$ (they agree on the cycles that are not internal to $H$) and $V(\mathcal{I}_{H_1}(C_1))=V(\mathcal{I}_{H_1}(C_2))$ (they use the same vertices of $H$ for the internal cycles). Because we are dealing with complete bipartite graphs, we have that $\mathcal{I}_{H_1}(C)$, which consists of some combination of disjoint edges and even cycles, uses all the vertices of some $K_{t,t}$. In particular, as we look over the cycle decompositions in some equivalence class $\mathcal{C}'$, we see that all non-internal cycles of the $K_{t,t}$ are fixed. Furthermore, the set $\overline{\mathcal{C}}$ of internal cycles that vary across all cycle decompositions in $\mathcal{C}'$ contains all cycle decompositions using the vertices of the $K_{t,t}$. Then, if $t\ge2$, we have that

\[
\sum_{C\in \mathcal{C}'}x^{u(C)}2^{\lng(C)}(-1)^{\cy(C)}\prod_{v\in C}\frac1{\deg(v)}=
f(x)\sum_{\overline{C}\in\overline{\mathcal{C}}}2^{\lng(\overline{C})}(-1)^{\cy(\overline{C})}=0,
\]
where $f(x)$ is the portion of the expression that is constant across all of the cycle decompositions in $\mathcal{C}'$ (which corresponds to the contributions of the fixed non-internal cycles as well as the isolated vertices and degrees which are constant across all $C \in \mathcal{C}'$). By an application of Proposition~\ref{prop:Ktt for A}, the sum over $\overline{\mathcal{C}}$ is $0$.

So, the only equivalence classes that make a contribution to $p(x)$ as shown in Proposition~\ref{prop:allC} are ones where $t \leq 1$. That is, cycle decompositions for which $\mathcal{I}(H_1)$ is either empty or a single edge.

Now that we have this for $H_1$, we can sequentially apply the same  argument for the remaining $H_i$. The only thing to be checked is if we hold the portion of a cycle decomposition $C$ which is not part of $\mathcal{I}_{H_i}(C)$ fixed, that the collection of all other cycle decompositions in the same group will collectively give all ways to decompose the complete bipartite graph corresponding with $\mathcal{I}_{H_i}(C)$. This follows since the $H_i$ are edge disjoint, and so any perturbation internally would not have been removed in an earlier round.
\end{proof}

\subsection{Ornamented binary trees and words}
We now construct families of graphs, where each familiy consists of one tree and several non-trees. The basic underlying structure will be full binary trees.

\begin{definition}
A \emph{full} binary tree is a rooted tree where every vertex has either zero or two children.
\end{definition}

We now modify our binary tree by ``turning every $K_{1,2}$ into a $K_{p,q}$'', as shown in Figure~\ref{fig:oneblock}. More precisely we have the following.

\begin{definition}
Given a full binary tree $T$ and $p\ge1$, $q\ge2$, the \emph{$(p,q)$ ornamented $T$} is the graph resulting from taking \emph{every} vertex (``$\triangle$'') with two children (``${+}$'' and ``${-}$'') and then forming a $K_{p,q}$. The part with size $p$ is formed by adding $p-1$ new vertices together with ``$\triangle$'' and the part with size $q$ is formed by adding $q-2$ new vertices together with ``${+}$'' and ``${-}$''.
\end{definition}

\begin{figure}[!htb]
\centering
\begin{tabular}{ccc}
\begin{tabular}{c}
\begin{tikzpicture}[yscale=0.6]
\tikzstyle{vertex}=[inner sep=2pt,thick,shape=circle,draw=black,fill=white];

\node[vertex] (l1) at (0.25,0) {};
\node at (0.25,0.5) {$\triangle$};

(1.25,0.75) rectangle (1.75,2.25);
\node[vertex] (r1) at (0,-1.5) {};
\node[vertex] (r2) at (0.5,-1.5) {};
\node at (0,-2) {${-}$};
\node at (0.5,-2) {${+}$};

\draw[thick] (r1)--(l1)--(r2);

\end{tikzpicture}
\end{tabular}
&$\longrightarrow$&
\begin{tabular}{c}
\begin{tikzpicture}[yscale=0.6]
\tikzstyle{vertex}=[inner sep=2pt,thick,shape=circle,draw=black,fill=white];

\draw[ultra thick, fill=black!10!white, draw = black!30!white,rounded corners] (0.5,0.25) rectangle (2,-0.25);
\node[vertex] (l1) at (0.25,0) {};
\node[vertex,fill=black!20!white] (l2) at (0.75,0) {};
\node[vertex,fill=black!20!white] (l3) at (1.75,0) {};
\node at (1.25,0) {$\cdots$};
\node at (0.25,0.5) {$\triangle$};

\draw[ultra thick, fill=black!10!white, draw = black!30!white,rounded corners] 
    (0.75,-1.25) rectangle (2.25,-1.75);
\node[vertex] (r1) at (0,-1.5) {};
\node[vertex] (r2) at (0.5,-1.5) {};
\node[vertex,fill=black!20!white] (r3) at (1,-1.5) {};
\node[vertex,fill=black!20!white] (r4) at (2,-1.5) {};
\node at (1.5,-1.5) {$\cdots$};
\node at (0,-2) {${-}$};
\node at (0.5,-2) {${+}$};

\draw[thick] (r1)--(l1)--(r2) (l2)--(r3)--(l3)--(r4)--(l2)  (r3)--(l1)--(r4) (l2)--(r1)--(l3) (l2)--(r2)--(l3);

\node at (1.25,0.75) {$p{-}1$};
\node at (1.5,-2.25) {$q{-}2$};
\end{tikzpicture}
\end{tabular}
\end{tabular}

\caption{Rules for creating ``blocks'' in an ornamented binary tree. In the full tree, there will sometimes be additional edges on ``$\triangle$'', ``${-}$'', or ``${+}$'' to an adjacent block.}
\label{fig:oneblock}
\end{figure}
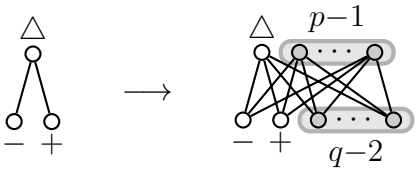

An example of a binary tree and three corresponding ornamented binary trees is shown in Figure~\ref{fig:ornamented}.

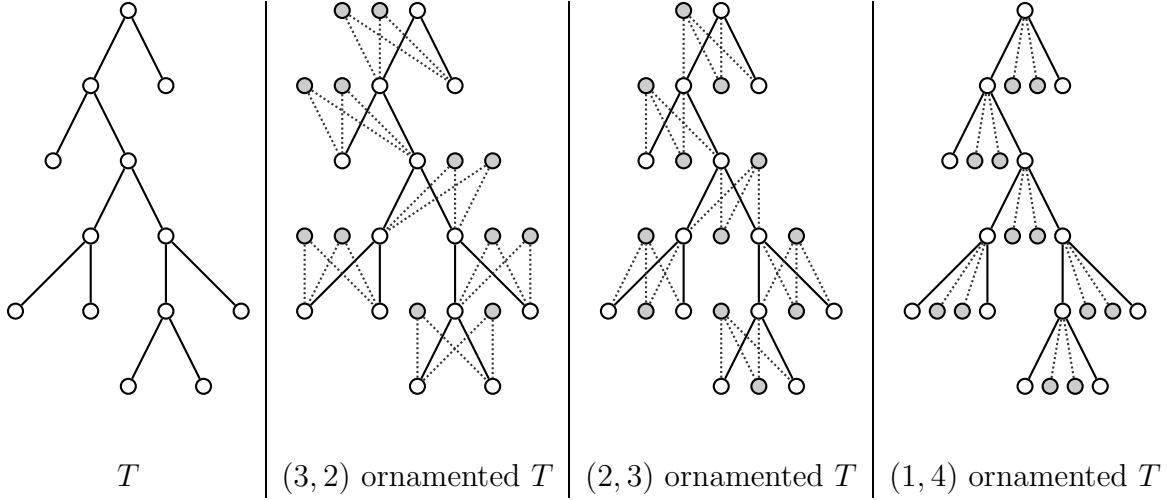
\begin{figure}
\centering
\begin{tabular}{c|c|c|c}
\begin{tikzpicture}[scale=0.5]
\tikzstyle{vertex}=[inner sep=2pt,thick,shape=circle,draw=black];
\tikzstyle{ornament}=[inner sep=2pt,thick,shape=circle,draw=black,fill=black!20!white];
\node[vertex] (A1) at ( 0, 0) {};
\node[vertex] (B1) at (-1,-2) {};
\node[vertex] (B2) at ( 1,-2) {};
\node[vertex] (C1) at (-2,-4) {};
\node[vertex] (C2) at ( 0,-4) {};
\node[vertex] (D1) at (-1,-6) {};
\node[vertex] (D2) at ( 1,-6) {};
\node[vertex] (E1) at (-3,-8) {};
\node[vertex] (E2) at (-1,-8) {};
\node[vertex] (E3) at ( 1,-8) {};
\node[vertex] (E4) at ( 3,-8) {};

\draw[thick] (A1)--(B1) (A1)--(B2) (B1)--(C1) (B1)--(C2) (C2)--(D1) (C2)--(D2) (D1)--(E1) (D1)--(E2) (D2)--(E3) (D2)--(E4);

\node[vertex] (F1) at ( 0,-10) {};
\node[vertex] (F2) at ( 2,-10) {};
\draw[thick] (E3)--(F1) (E3)--(F2);
\end{tikzpicture}

&
\begin{tikzpicture}[scale=0.5]
\tikzstyle{vertex}=[inner sep=2pt,thick,shape=circle,draw=black];
\tikzstyle{ornament}=[inner sep=2pt,thick,shape=circle,draw=black,fill=black!20!white];
\node[vertex] (A1) at ( 0, 0) {};
\node[vertex] (B1) at (-1,-2) {};
\node[vertex] (B2) at ( 1,-2) {};
\node[vertex] (C1) at (-2,-4) {};
\node[vertex] (C2) at ( 0,-4) {};
\node[vertex] (D1) at (-1,-6) {};
\node[vertex] (D2) at ( 1,-6) {};
\node[vertex] (E1) at (-3,-8) {};
\node[vertex] (E2) at (-1,-8) {};
\node[vertex] (E3) at ( 1,-8) {};
\node[vertex] (E4) at ( 3,-8) {};

\draw[thick] (A1)--(B1) (A1)--(B2) (B1)--(C1) (B1)--(C2) (C2)--(D1) (C2)--(D2) (D1)--(E1) (D1)--(E2) (D2)--(E3) (D2)--(E4);

\node[ornament] (a1) at (-1, 0) {};
\node[ornament] (a2) at (-2, 0) {};
\node[ornament] (b1) at (-2,-2) {};
\node[ornament] (b2) at (-3,-2) {};
\node[ornament] (c1) at ( 1,-4) {};
\node[ornament] (c2) at ( 2,-4) {};
\node[ornament] (d1) at (-3,-6) {};
\node[ornament] (d2) at (-2,-6) {};
\node[ornament] (d3) at ( 2,-6) {};
\node[ornament] (d4) at ( 3,-6) {};

\draw[ thick,color=black!80!white, densely dotted]
(a1)--(B1) (a1)--(B2) (a2)--(B1) (a2)--(B2)
(b1)--(C1) (b1)--(C2) (b2)--(C1) (b2)--(C2)
(c1)--(D1) (c1)--(D2) (c2)--(D1) (c2)--(D2)
(d1)--(E1) (d1)--(E2) (d2)--(E1) (d2)--(E2) (d3)--(E3) (d3)--(E4) (d4)--(E3) (d4)--(E4);

\node[vertex] (F1) at ( 0,-10) {};
\node[vertex] (F2) at ( 2,-10) {};
\draw[thick] (E3)--(F1) (E3)--(F2);
\node[ornament] (e1) at (0,-8) {};
\node[ornament] (e2) at (2,-8) {};
\draw[ thick,color=black!80!white, densely dotted]
(F1)--(e1)--(F2) (F1)--(e2)--(F2);

\end{tikzpicture}

&
\begin{tikzpicture}[scale=0.5]
\tikzstyle{vertex}=[inner sep=2pt,thick,shape=circle,draw=black];
\tikzstyle{ornament}=[inner sep=2pt,thick,shape=circle,draw=black,fill=black!20!white];
\node[vertex] (A1) at ( 0, 0) {};
\node[vertex] (B1) at (-1,-2) {};
\node[vertex] (B2) at ( 1,-2) {};
\node[vertex] (C1) at (-2,-4) {};
\node[vertex] (C2) at ( 0,-4) {};
\node[vertex] (D1) at (-1,-6) {};
\node[vertex] (D2) at ( 1,-6) {};
\node[vertex] (E1) at (-3,-8) {};
\node[vertex] (E2) at (-1,-8) {};
\node[vertex] (E3) at ( 1,-8) {};
\node[vertex] (E4) at ( 3,-8) {};

\draw[thick] (A1)--(B1) (A1)--(B2) (B1)--(C1) (B1)--(C2) (C2)--(D1) (C2)--(D2) (D1)--(E1) (D1)--(E2) (D2)--(E3) (D2)--(E4);

\node[ornament] (a1) at (-1, 0) {};
\node[ornament] (a2) at ( 0,-2) {};
\node[ornament] (b1) at (-2,-2) {};
\node[ornament] (b2) at (-1,-4) {};
\node[ornament] (c1) at ( 1,-4) {};
\node[ornament] (c2) at ( 0,-6) {};
\node[ornament] (d1) at (-2,-6) {};
\node[ornament] (d2) at (-2,-8) {};
\node[ornament] (d3) at ( 2,-6) {};
\node[ornament] (d4) at ( 2,-8) {};

\draw[ thick,color=black!80!white, densely dotted]
(a1)--(B1) (a1)--(B2) (a2)--(a1) (a2)--(A1)
(b1)--(C1) (b1)--(C2) (b2)--(b1) (b2)--(B1)
(c1)--(D1) (c1)--(D2) (c2)--(c1) (c2)--(C2)
(d1)--(E1) (d1)--(E2) (d2)--(d1) (d2)--(D1) (d3)--(E3) (d3)--(E4) (d4)--(d3) (d4)--(D2);

\node[vertex] (F1) at ( 0,-10) {};
\node[vertex] (F2) at ( 2,-10) {};
\draw[thick] (E3)--(F1) (E3)--(F2);
\node[ornament] (e1) at (0,-8) {};
\node[ornament] (e2) at (1,-10) {};
\draw[ thick,color=black!80!white, densely dotted]
(F1)--(e1)--(F2) (e1)--(e2)--(E3);

\end{tikzpicture}
&

\begin{tikzpicture}[scale=0.5]
\tikzstyle{vertex}=[inner sep=2pt,thick,shape=circle,draw=black];
\tikzstyle{ornament}=[inner sep=2pt,thick,shape=circle,draw=black,fill=black!20!white];
\node[vertex] (A1) at ( 0, 0) {};
\node[vertex] (B1) at (-1,-2) {};
\node[vertex] (B2) at ( 1,-2) {};
\node[vertex] (C1) at (-2,-4) {};
\node[vertex] (C2) at ( 0,-4) {};
\node[vertex] (D1) at (-1,-6) {};
\node[vertex] (D2) at ( 1,-6) {};
\node[vertex] (E1) at (-3,-8) {};
\node[vertex] (E2) at (-1,-8) {};
\node[vertex] (E3) at ( 1,-8) {};
\node[vertex] (E4) at ( 3,-8) {};

\draw[thick] (A1)--(B1) (A1)--(B2) (B1)--(C1) (B1)--(C2) (C2)--(D1) (C2)--(D2) (D1)--(E1) (D1)--(E2) (D2)--(E3) (D2)--(E4);

\node[ornament] (a1) at ({-1/3},-2) {};
\node[ornament] (a2) at ({ 1/3},-2) {};
\node[ornament] (b1) at ({-2/3},-4) {};
\node[ornament] (b2) at ({-4/3},-4) {};
\node[ornament] (c1) at ({-1/3},-6) {};
\node[ornament] (c2) at ({ 1/3},-6) {};
\node[ornament] (d1) at ({-7/3},-8) {};
\node[ornament] (d2) at ({-5/3},-8) {};
\node[ornament] (d3) at ({5/3},-8) {};
\node[ornament] (d4) at ({7/3},-8) {};

\draw[ thick,color=black!80!white, densely dotted]
(a1)--(A1)--(a2)
(b1)--(B1)--(b2)
(c1)--(C2)--(c2)
(d1)--(D1)--(d2)
(d3)--(D2)--(d4);

\node[vertex] (F1) at ( 0,-10) {};
\node[vertex] (F2) at ( 2,-10) {};
\draw[thick] (E3)--(F1) (E3)--(F2);
\node[ornament] (e1) at ({2/3},-10) {};
\node[ornament] (e2) at ({4/3},-10) {};
\draw[ thick,color=black!80!white, densely dotted]
(e2)--(E3)--(e1);

\end{tikzpicture}

\\[20pt]
$T$&$(3,2)$ ornamented $T$&$(2,3)$ ornamented $T$&$(1,4)$ ornamented $T$
\end{tabular}
\caption{A full binary tree $T$ and several associated ornamented trees (new vertices and edges are shaded to highlight the underlying tree).}
\label{fig:ornamented}
\end{figure}

\begin{observation}
A $(p,q)$ ornamented $T$ is a tree if and only if $p=1$.
\end{observation}

To help facilitate our conversation we want to find a way to describe our full binary tree $T$ by a word, which we will denote $W_T$. We will do this by the use of the following iterative rules starting from the root, also shown pictorially in Figure~\ref{fig:word}.

\begin{figure}[!htb]
\centering
\begin{tabular}{|c|c|c|} \hline
No subtrees & One subtree & Two subtrees \\ \hline &&\\[-7pt]
\begin{tabular}{c}
\begin{tikzpicture}[scale=0.5]
\tikzstyle{vertex}=[inner sep=2pt,thick,shape=circle,draw=black];

\draw[thick,densely dotted,color=white,rounded corners]
 (-1.25,-2) -- (-.75,-2) -- (-0.25,-4.5) -- (-1.25,-4.5);

\node[vertex]  (A) at ( 0, 0) {};
\node[vertex] (B1) at (-1.25,-2) {};
\node[vertex] (B2) at ( 1.25,-2) {};
\draw[thick] (B1)--(A)--(B2);
\end{tikzpicture}
\end{tabular}
&
\begin{tabular}{c}
\begin{tikzpicture}[scale=0.5]
\tikzstyle{vertex}=[inner sep=2pt,thick,shape=circle,draw=black,fill=white];

\draw[thick,densely dotted,fill=black!15!white,rounded corners]
 (-1.25,-2) -- (-.75,-2) -- (-0.25,-4.5) -- (-2.25,-4.5) -- (-1.75,-2) -- (-1.25,-2);
 \node at (-1.25,-3.75) {$T_1$};

\node[vertex]  (A) at ( 0, 0) {};
\node[vertex] (B1) at (-1.25,-2) {};
\node[vertex] (B2) at ( 1.25,-2) {};
\draw[thick] (B1)--(A)--(B2);

\end{tikzpicture}
\end{tabular}
&
\begin{tabular}{c}
\begin{tikzpicture}[scale=0.5]
\tikzstyle{vertex}=[inner sep=2pt,thick,shape=circle,draw=black,fill=white];

\draw[thick,densely dotted,fill=black!15!white,rounded corners,shift={(2.5,0)}]
 (-1.25,-2) -- (-.75,-2) -- (-0.25,-4.5) -- (-2.25,-4.5) -- (-1.75,-2) -- (-1.25,-2);
 \node at (1.25,-3.75) {$T_2$};

\draw[thick,densely dotted,fill=black!15!white,rounded corners]
 (-1.25,-2) -- (-.75,-2) -- (-0.25,-4.5) -- (-2.25,-4.5) -- (-1.75,-2) -- (-1.25,-2);
 \node at (-1.25,-3.75) {$T_1$};

\node[vertex]  (A) at ( 0, 0) {};
\node[vertex] (B1) at (-1.25,-2) {};
\node[vertex] (B2) at ( 1.25,-2) {};
\draw[thick] (B1)--(A)--(B2);

\end{tikzpicture}
\end{tabular}

\\ \hline
$e$&$SW_{T_1}$&$D((W_{T_1}){\otimes}(W_{T_2}))$ \\ \hline
\end{tabular}
\caption{Rules for forming the word for $T$.}
\label{fig:word}
\end{figure}

\begin{itemize}
\item \emph{End} ($e$). Neither child has a subtree attached. Then $W_T=e$.
\item \emph{Single ($S$)}. Single child has a subtree attached, which we denote $T_1$. Then $W_T=SW_{T_1}$.
\item \emph{Double ($D$)}. Both children have subtrees attached, which we denote $T_1$ and $T_2$, respectively.  Then $W_T=D((W_{T_1}){\otimes}(W_{T_2}))$.
\end{itemize}

Finally, the \emph{extended word of $T$} is given by $iW_T$ (where we think $i$ stands for \emph{initial}). As an example, the extended word for the tree $T$ in Figure~\ref{fig:ornamented} is $iSSD((e){\otimes}(Se))$.

Since every possible situation that can occur in a full binary tree has been accounted for when making the word, every full binary tree has an extended word. We note in passing that words for a given tree may not be unique, e.g.\ swapping the words from branching will not change the word so that $T$ from Figure~\ref{fig:ornamented} could also be $iSSD((Se){\otimes}(e))$. This will not be an issue for our results.

\subsection{Equivalence of characteristic polynomials}
With the description of our objects of interest in place we are ready to state our main result for the normalized adjacency matrix $\mathcal{A}$.

\begin{theorem}\label{thm:Agraphs}
Let $T$ be a full binary tree, and let $p,p'\ge1$, $q,q'\ge2$ with $p+q=p'+q'$. Then the $(p,q)$ ornamented $T$ is cospectral with the $(p',q')$ ornamented $T$ for $\mathcal{A}$.
\end{theorem}

As an example, the three ornamented trees in Figure~\ref{fig:ornamented} are cospectral. The proof will have two steps. First, we show how to compute the characteristic polynomial for these graphs from a product of matrices based on the extended word. Second, we show that the result of this product, and hence characteristic polynomials, are the same for both graphs. From Theorem~\ref{thm:Agraphs} we will get the following.

\begin{corollary}
Let $\alpha\ge3$ be fixed. Then there exists exponentially many spectral faux trees on $n=1+\alpha k$ (as a function of $k$) vertices.
\end{corollary}
\begin{proof}
The number of full binary trees on $n=1+2k$ grows exponentially (they are counted by the Catalan numbers \cite{stanley}). For each such full binary tree $T$, we construct the $(1,\alpha)$ ornamented $T$ and the $(2,\alpha-1)$ ornamented $T$. The first is a tree, while the second is not a tree (since it has a $4$-cycle). Both graphs have $n=1+\alpha k$ vertices, and so we have created as many spectral faux trees on $n=1+\alpha k$ as there are full binary trees on $n=1+2k$ vertices.

For completeness, we should verify that no two of the $(2,\alpha-1)$ ornamented trees are isomorphic. To do this we first observe that we can identify all of the $K_{2,\alpha-1}$ blocks in the graph and how they are connected to one another. If the tree $T$ has any vertex with two subtrees off the children, we can identify such vertices in the ornamented tree, e.g.\ these would be blocks with two blocks connected to the same part. In particular, with such a block we can identify the ``parent'' vertex, and then using that, trace back to find the root vertex of $T$, and more generally construct $T$. This covers all but one binary tree $T$, namely the one without any branching to two subtrees. For this final one there would not be any block with two blocks connected to the same part, which handles the last case as well.
\end{proof}

\begin{proof}[Proof of Theorem~\ref{thm:Agraphs}]
Since a $(p,q)$ ornamented $T$ can be described as a collection of edge disjoint complete graphs, we can use Lemma~\ref{lem:usefulC} to simplify the computation of the characteristic polynomial to considering cycle decompositions where any \emph{internal cycles} of each $K_{p,q}$ block must consist of at most one edge. In addition, there cannot be any cycles which involve edges from several of the $K_{p,q}$ (if there were this would give a cycle in the underlying complete full binary tree $T$, which does not exist). So, when computing the characteristic polynomials for our family of graphs, we can focus on finding all possible ways of picking at most one edge in each block, and for each such way, determining the contribution to the characteristic polynomial.

We begin by identifying the different ways we can select (or not select) a single edge in any one of our complete bipartite graphs. Referencing the notation in Figure~\ref{fig:oneblock}, we compile the polynomial contributions in Table~\ref{tab:contributions}. 

Each $K_{p,q}$ in the $(p,q)$ ornamented $T$ has $p+q$ vertices, however some vertices might be shared between two such $K_{p,q}$ (e.g.\ ``$\triangle$'' at one level with either ``${+}$'' or ``${-}$'' in the level above). We will deal with this situation by treating each new $K_{p,q}$ as introducing $p+q-1$ vertices (the non-``$\triangle$''). So, the polynomial contributions in Table~\ref{tab:contributions} all have an extra factor of $x^{p+q-3}$ which we have factored out (the exponent reflecting that we have $p+q-1$ vertices and then two of them will be involved in an edge).

\begin{table}[!htb]
\centering
\begin{tabular}{|c|c|c|}\hline
\begin{tabular}{c}
Relationship\\ with $\{\triangle,{+},{-}\}$
\end{tabular}&
Visual
&Polynomial contribution\\ \hline
no edge &
\begin{tabular}{c}
\begin{tikzpicture}[scale=0.75,transform shape]
\tikzstyle{vertex}=[inner sep=2pt,thick,shape=circle,draw=black,fill=white];

\draw[ultra thick, fill=black!10!white, draw = black!30!white,rounded corners] (0.5,0.25) rectangle (2,-0.25);
\node[vertex] (l1) at (0.25,0) {};
\node[vertex] (l2) at (0.75,0) {};
\node[vertex] (l3) at (1.75,0) {};
\node at (1.25,0) {$\cdots$};
\node at (0.25,0.3) {$\triangle$};

\draw[ultra thick, fill=black!10!white, draw = black!30!white,rounded corners] (0.75,-0.75) rectangle (2.25,-1.25);
\node[vertex] (r1) at (0,-1) {};
\node[vertex] (r2) at (0.5,-1) {};
\node[vertex] (r3) at (1,-1) {};
\node[vertex] (r4) at (2,-1) {};
\node at (1.5,-1) {$\cdots$};
\node at (0,-1.3) {${-}$};
\node at (0.5,-1.3) {${+}$};

\draw[thick,densely dotted] (r1)--(l1)--(r2) (l2)--(r3)--(l3)--(r4)--(l2)  (r3)--(l1)--(r4) (l2)--(r1)--(l3) (l2)--(r2)--(l3);
\end{tikzpicture}
\end{tabular}
& $x^2$ \\ \hline
$\emptyset$&

\begin{tabular}{c}
\begin{tikzpicture}[scale=0.75,transform shape]
\tikzstyle{vertex}=[inner sep=2pt,thick,shape=circle,draw=black,fill=white];

\draw[ultra thick, fill=black!10!white, draw = black!30!white,rounded corners] (0.5,0.25) rectangle (2,-0.25);
\node[vertex] (l1) at (0.25,0) {};
\node[vertex] (l2) at (0.75,0) {};
\node[vertex] (l3) at (1.75,0) {};
\node at (1.25,0) {$\cdots$};
\node at (0.25,0.3) {$\triangle$};

\draw[ultra thick, fill=black!10!white, draw = black!30!white,rounded corners] (0.75,-0.75) rectangle (2.25,-1.25);
\node[vertex] (r1) at (0,-1) {};
\node[vertex] (r2) at (0.5,-1) {};
\node[vertex] (r3) at (1,-1) {};
\node[vertex] (r4) at (2,-1) {};
\node at (1.5,-1) {$\cdots$};
\node at (0,-1.3) {${-}$};
\node at (0.5,-1.3) {${+}$};

\draw[thick,densely dotted] (r1)--(l1)--(r2) (l2)--(r3)--(l3)--(r4)--(l2)  (r3)--(l1)--(r4) (l2)--(r1)--(l3) (l2)--(r2)--(l3);
\draw[line width=3pt, color=red, opacity=0.7] (l2)--(r3);
\draw[line width=1pt] (l2)--(r3);
\end{tikzpicture}
\end{tabular}
&$-\dfrac{(p-1)(q-2)}{pq}$\\ \hline

${+}$&
\begin{tabular}{c}
\begin{tikzpicture}[scale=0.75,transform shape]
\tikzstyle{vertex}=[inner sep=2pt,thick,shape=circle,draw=black,fill=white];

\draw[ultra thick, fill=black!10!white, draw = black!30!white,rounded corners] (0.5,0.25) rectangle (2,-0.25);
\node[vertex] (l1) at (0.25,0) {};
\node[vertex] (l2) at (0.75,0) {};
\node[vertex] (l3) at (1.75,0) {};
\node at (1.25,0) {$\cdots$};
\node at (0.25,0.3) {$\triangle$};

\draw[ultra thick, fill=black!10!white, draw = black!30!white,rounded corners] (0.75,-0.75) rectangle (2.25,-1.25);
\node[vertex] (r1) at (0,-1) {};
\node[vertex] (r2) at (0.5,-1) {};
\node[vertex] (r3) at (1,-1) {};
\node[vertex] (r4) at (2,-1) {};
\node at (1.5,-1) {$\cdots$};
\node at (0,-1.3) {${-}$};
\node at (0.5,-1.3) {${+}$};

\draw[thick,densely dotted] (r1)--(l1)--(r2) (l2)--(r3)--(l3)--(r4)--(l2)  (r3)--(l1)--(r4) (l2)--(r1)--(l3) (l2)--(r2)--(l3);
\draw[line width=3pt, color=red, opacity=0.7] (r2)--(l2);
\draw[line width=1pt] (l2)--(r2);
\end{tikzpicture}
\end{tabular}
&$-\dfrac{(p-1)}{\deg({+})q}$\\ \hline

${-}$&
\begin{tabular}{c}
\begin{tikzpicture}[scale=0.75,transform shape]
\tikzstyle{vertex}=[inner sep=2pt,thick,shape=circle,draw=black,fill=white];

\draw[ultra thick, fill=black!10!white, draw = black!30!white,rounded corners] (0.5,0.25) rectangle (2,-0.25);
\node[vertex] (l1) at (0.25,0) {};
\node[vertex] (l2) at (0.75,0) {};
\node[vertex] (l3) at (1.75,0) {};
\node at (1.25,0) {$\cdots$};
\node at (0.25,0.3) {$\triangle$};

\draw[ultra thick, fill=black!10!white, draw = black!30!white,rounded corners] (0.75,-0.75) rectangle (2.25,-1.25);
\node[vertex] (r1) at (0,-1) {};
\node[vertex] (r2) at (0.5,-1) {};
\node[vertex] (r3) at (1,-1) {};
\node[vertex] (r4) at (2,-1) {};
\node at (1.5,-1) {$\cdots$};
\node at (0,-1.3) {${-}$};
\node at (0.5,-1.3) {${+}$};

\draw[thick,densely dotted] (r1)--(l1)--(r2) (l2)--(r3)--(l3)--(r4)--(l2)  (r3)--(l1)--(r4) (l2)--(r1)--(l3) (l2)--(r2)--(l3);
\draw[line width=3pt, color=red, opacity=0.7] (r1)--(l2);
\draw[line width=1pt] (l2)--(r1);
\end{tikzpicture}
\end{tabular}
&$-\dfrac{(p-1)}{\deg({-})q}$\\ \hline

$\triangle$&
\begin{tabular}{c}
\begin{tikzpicture}[scale=0.75,transform shape]
\tikzstyle{vertex}=[inner sep=2pt,thick,shape=circle,draw=black,fill=white];

\draw[ultra thick, fill=black!10!white, draw = black!30!white,rounded corners] (0.5,0.25) rectangle (2,-0.25);
\node[vertex] (l1) at (0.25,0) {};
\node[vertex] (l2) at (0.75,0) {};
\node[vertex] (l3) at (1.75,0) {};
\node at (1.25,0) {$\cdots$};
\node at (0.25,0.3) {$\triangle$};

\draw[ultra thick, fill=black!10!white, draw = black!30!white,rounded corners] (0.75,-0.75) rectangle (2.25,-1.25);
\node[vertex] (r1) at (0,-1) {};
\node[vertex] (r2) at (0.5,-1) {};
\node[vertex] (r3) at (1,-1) {};
\node[vertex] (r4) at (2,-1) {};
\node at (1.5,-1) {$\cdots$};
\node at (0,-1.3) {${-}$};
\node at (0.5,-1.3) {${+}$};

\draw[thick,densely dotted] (r1)--(l1)--(r2) (l2)--(r3)--(l3)--(r4)--(l2)  (r3)--(l1)--(r4) (l2)--(r1)--(l3) (l2)--(r2)--(l3);
\draw[line width=3pt, color=red, opacity=0.7] (l1)--(r3);
\draw[line width=1pt] (l1)--(r3);
\end{tikzpicture}
\end{tabular}
&$-\dfrac{(q-2)}{\deg(\triangle)p}$\\ \hline

$\triangle,{+}$&
\begin{tabular}{c}
\begin{tikzpicture}[scale=0.75,transform shape]
\tikzstyle{vertex}=[inner sep=2pt,thick,shape=circle,draw=black,fill=white];

\draw[ultra thick, fill=black!10!white, draw = black!30!white,rounded corners] (0.5,0.25) rectangle (2,-0.25);
\node[vertex] (l1) at (0.25,0) {};
\node[vertex] (l2) at (0.75,0) {};
\node[vertex] (l3) at (1.75,0) {};
\node at (1.25,0) {$\cdots$};
\node at (0.25,0.3) {$\triangle$};

\draw[ultra thick, fill=black!10!white, draw = black!30!white,rounded corners] (0.75,-0.75) rectangle (2.25,-1.25);
\node[vertex] (r1) at (0,-1) {};
\node[vertex] (r2) at (0.5,-1) {};
\node[vertex] (r3) at (1,-1) {};
\node[vertex] (r4) at (2,-1) {};
\node at (1.5,-1) {$\cdots$};
\node at (0,-1.3) {${-}$};
\node at (0.5,-1.3) {${+}$};

\draw[thick,densely dotted] (r1)--(l1)--(r2) (l2)--(r3)--(l3)--(r4)--(l2)  (r3)--(l1)--(r4) (l2)--(r1)--(l3) (l2)--(r2)--(l3);
\draw[line width=3pt, color=red, opacity=0.7] (l1)--(r2);
\draw[line width=1pt] (l1)--(r2);
\end{tikzpicture}
\end{tabular}
&$-\dfrac{1}{\deg(\triangle)\deg({+})}$\\ \hline

$\triangle,{-}$&
\begin{tabular}{c}
\begin{tikzpicture}[scale=0.75,transform shape]
\tikzstyle{vertex}=[inner sep=2pt,thick,shape=circle,draw=black,fill=white];

\draw[ultra thick, fill=black!10!white, draw = black!30!white,rounded corners] (0.5,0.25) rectangle (2,-0.25);
\node[vertex] (l1) at (0.25,0) {};
\node[vertex] (l2) at (0.75,0) {};
\node[vertex] (l3) at (1.75,0) {};
\node at (1.25,0) {$\cdots$};
\node at (0.25,0.3) {$\triangle$};

\draw[ultra thick, fill=black!10!white, draw = black!30!white,rounded corners] (0.75,-0.75) rectangle (2.25,-1.25);
\node[vertex] (r1) at (0,-1) {};
\node[vertex] (r2) at (0.5,-1) {};
\node[vertex] (r3) at (1,-1) {};
\node[vertex] (r4) at (2,-1) {};
\node at (1.5,-1) {$\cdots$};
\node at (0,-1.3) {${-}$};
\node at (0.5,-1.3) {${+}$};

\draw[thick,densely dotted] (r1)--(l1)--(r2) (l2)--(r3)--(l3)--(r4)--(l2)  (r3)--(l1)--(r4) (l2)--(r1)--(l3) (l2)--(r2)--(l3);
\draw[line width=3pt, color=red, opacity=0.7] (l1)--(r1);
\draw[line width=1pt] (l1)--(r1);
\end{tikzpicture}
\end{tabular}
&$-\dfrac{1}{\deg(\triangle)\deg({-})}$\\ \hline

\end{tabular}
\caption{The possible polynomial contributions of edges inside of each block. Each polynomial contribution has an extra $x^{p+q-3}$ from isolated vertices that has been pulled out.}
\label{tab:contributions}
\end{table}

Looking at the entries in Table~\ref{tab:contributions}, the $x^2$ term corresponds with the situation where no edge is used. For all other terms,  ``$-$'' terms come from the edge corresponding to one cycle (thus contributing a $(-1)$ in the $(-1)^{\cy(C)}$ in \eqref{prop:allC}). The numerator represents the number of such possible edges of a given form. The denominator represents the degree of the vertices involved in the edges. For this last point we note that $\deg(\triangle)$, $\deg({+})$, and $\deg({-})$ can have varying degrees depending on whether the vertices connect with another block.

There is one subtlety in the table, and that is for edges which involve ``$\triangle$''. As mentioned above, we introduce $p+q-1$ vertices in each $K_{p,q}$ and since an edge involving ``$\triangle$'' only uses one of the non-``$\triangle$'' vertices it would seem that there should be an extra factor of $x^{p+q-2}$ instead of $x^{p+q-3}$. However, in this case, we also have that in the other $K_{p,q}$ which involved that ``$\triangle$'' (either as a ``${+}$'' or ``${-}$'') that ``$\triangle$'' made a contribution of $x$ which should not have occurred. So, what we need to do is compensate, which we do by adding an extra term of $1/x$. Combining this with the $x^{p+q-2}$ term brings us to $x^{p+q-3}$, as indicated in the caption of Table~\ref{tab:contributions}. 

Given the polynomial contribution of each block, the contribution for the whole cycle decomposition is now the product of all of the individual contributions.  The main difficulty remaining is that the choice of edge in one $K_{p,q}$ can have an effect on the choice of edge in an adjacent $K_{p,q}$ and this must be taken into account. To do this we think about forming a walk between blocks where the edge used in the current blocks informs our choice of which edges can be selected in the next block. As an example, if we have an edge using the vertex ``${-}$'', then if a block connects on that vertex, that block cannot use any edge involving ``$\triangle$''; otherwise there is no restriction. A similar statement holds for the vertex ``${+}$''.

We now introduce a collection of matrices, $i_{p,q}$, $S_{p,q}$, $D_{p,q}$, and $e_{p,q}$, that capture both the contribution from choice of edge and allowable transitions to the next block. The indexing of the matrices is the same as that of Table~\ref{tab:contributions}, and our choice of names (deliberately) bears a resemblance to extended words. For the following we will let $\mathbf{e}=[1\,1\,1\,1\,1\,1\,1]$ and $\mathbf{t}=[1\,1\,1\,1\,0\,0\,0]$ (which represent transitioning to all possible choices for edges in the next step, and transitioning to choosing edges that don't involve ``$\triangle$'', respectively), while $\otimes$ will represent the tensor product.
\begin{gather*}
\frac{i_{p,q}}{x}=\begin{bmatrix}
1&1&1&1&\tfrac{p+q}q&\tfrac{p+q}q&\tfrac{p+q}q
\end{bmatrix}
\\[10pt]
\frac{S_{p,q}}{x^{p+q-3}}=
\begin{bmatrix}
x^2&0&0&0&0&0&0\\
0&-\tfrac{(p-1)(q-2)}{pq}&0&0&0&0&0\\
0&0&-\tfrac{p-1}{pq}&0&0&0&0\\
0&0&0&-\frac{p-1}{q(p+q)}&0&0&0\\
0&0&0&0&-\frac{q-2}{p(p+q)}&0&0\\
0&0&0&0&0&-\frac{1}{p(p+q)}&0\\
0&0&0&0&0&0&-\frac{1}{(p+q)^2}
\end{bmatrix}
\begin{bmatrix}
\mathbf{e}\\
\mathbf{e}\\
\mathbf{e}\\
\mathbf{t}\\
\mathbf{e}\\
\mathbf{e}\\
\mathbf{t}
\end{bmatrix}
\\[10pt]
\frac{D_{p,q}}{x^{p+q-3}}=
\begin{bmatrix}
x^2&0&0&0&0&0&0\\
0&-\tfrac{(p-1)(q-2)}{pq}&0&0&0&0&0\\
0&0&-\tfrac{p-1}{q(p+q)}&0&0&0&0\\
0&0&0&-\frac{p-1}{q(p+q)}&0&0&0\\
0&0&0&0&-\frac{q-2}{p(p+q)}&0&0\\
0&0&0&0&0&-\frac{1}{(p+q)^2}&0\\
0&0&0&0&0&0&-\frac{1}{(p+q)^2}
\end{bmatrix}
\begin{bmatrix}
\mathbf{e} \otimes \mathbf{e}\\
\mathbf{e} \otimes \mathbf{e}\\
\mathbf{t} \otimes \mathbf{e}\\
\mathbf{e} \otimes \mathbf{t}\\
\mathbf{e} \otimes \mathbf{e}\\
\mathbf{t} \otimes \mathbf{e}\\
\mathbf{e} \otimes \mathbf{t}
\end{bmatrix}
\\[10pt]
\frac{e_{p,q}}{x^{p+q-3}}=
\begin{bmatrix}
x^2&-\tfrac{(p-1)(q-2)}{pq} &
-\tfrac{p-1}{pq} &
-\frac{p-1}{pq} &
-\frac{q-2}{p(p+q)} &
-\frac{1}{p(p+q)} &
-\frac{1}{p(p+q)}
\end{bmatrix}^T
\end{gather*}

\begin{claim}
Given an extended word for the full binary tree $T$, the characteristic polynomial of the $(p,q)$ ornamented $T$ is found by taking the extended word for $T$ and replacing the letters $i$, $S$, $D$, and $e$ with the matrices $i_{p,q}$, $S_{p,q}$, $D_{p,q}$, and $e_{p,q}$, respectively, and carrying out the matrix multiplication. For either a word or extended word $W$ we denote the corresponding matrix product as $W_{p,q}$.
\end{claim}

To verify this claim we first show that for non-extended words that the resulting column vector is a polynomial whose entries reflect the contribution to the characteristic polynomial starting from where we assume the root has degree $p+q$ instead of degree $q$ and the entries indexed with the seven choices that are given in Table~\ref{tab:contributions}. By construction this holds for $e_{p,q}$ (which we can treat as our base case $W=e$).

Now suppose it holds for a tree with word $W$, and consider the tree with word $SW$. This translates to the product $S_{p,q}W_{p,q}$. The matrix $S_{p,q}$ works by first taking the contribution of each possible choice and then for each possible legal transition multiplying the corresponding entry of $W_{p,q}$, and summing the possible results. So the result also holds for $SW$. [By choice of matrix we have that $S_{p,q}$ represents connecting to the subtree via ``${-}$''.]

Now, suppose it holds for the trees with words $W$ and $W'$, and consider the tree with word $D((W){\otimes}(W'))$. This translates to the product $D_{p,q}(W_{p,q} \otimes W_{p,q}')$. The key is to note that tensor products takes all possible pairs of entries of $f_1\otimes f_2$, and in particular that it can be though of as a Cartesian product in lexicographical order for the 
\[
(\text{edge choice of block through ``${+}$''},\text{edge choice of block through ``${-}$''}).
\]
Now, as before, the multiplication by $D_{p,q}$ works by first taking the contribution of each possible choice, and then for each possible legal transition to a pair multiplying the corresponding entry of $W_{p,q}\otimes W_{p,q}'$, and summing the possible results. So the result also holds for $D((W){\otimes}(W'))$. [By choice of matrix we have that $D_{p,q}$ represents connecting to the subtree via ``${+}$'' using $W_{p,q}$  and the subtree via ``${-}$'' using $W_{p,q}'$. This also explains the format of the rows in the definition $D_{p,q}$, namely $\mathbf{t}\otimes\mathbf{e}$ is used when transitioning from an edge which uses ``${+}$'' and so cannot transition to a pair where the first entry uses ``$\triangle$''; similarly $\mathbf{e}\otimes\mathbf{t}$ is used when transitioning from an edge which uses ``${-}$'' and so cannot transition to a pair where the second entry uses ``$\triangle$''.]

This completes the induction for words, and leaves us to handle the extended words. In particular, the addition of the $i_{p,q}$ at the start must handle three issues: (1) we must produce a single polynomial instead of a column of polynomials; (2) we must take into account the initial vertex's contribution; (3) we must correct (where needed) that the root has degree $q$ and not degree $p+q$. For the third item, this will only affect the column entries which used ``$\triangle$'' in an edge, in which case we can multiply by $(p+q)/q$ to correct the product of degrees. For the second item, this is solved by multiplying \emph{all} of the column entries by $x$ (recall that in our setup, we always think of our blocks as introducing $p+q-1$ new vertices, and so the root and its contribution also needs to be included). We mention in passing that it would seem that we should only multiply those entries representing a situation not involving ``$\triangle$''; however as discussed earlier, there was already a correction term built into those cases. Finally, to add the results, we can multiply by an all-$1$s row vector. Combining all three of these operations gives the matrix $i_{p,q}$, and this finishes the claim.

With the claim established, the last ingredient for the proof is the following matrix.
\[
U_{p,q}=\begin{bmatrix}
1 & 0 & 0 & 0 & 0 & 0 & 0 \\
0 & 1 & 0 & 0 & a & 0 & 0 \\
0 & 0 & 1 & 0 & b & c & 0 \\
0 & 0 & 0 & 1 & b & 0 & c \\
0 & 0 & 0 & 0 & d & 0 & 0 \\
0 & 0 & 0 & 0 & e & f & 0 \\
0 & 0 & 0 & 0 & e & 0 & f
\end{bmatrix}
\]
where
\begin{align*}
a &= \frac{(p+q)(p-1)(q-2)}{pq(p+q-3)}
&
b&=\frac{(p+q)(p-1)^2}{pq(p+q-1)(p+q-3)}
&
c&=\frac{(p+q)(p-1)}{p(p+q-1)}\\
d&=\frac{q-2}{p(p+q-3)}
&
e&=\frac{p-1}{p(p+q-1)(p+q-3)}
&
f&=\frac{q}{p(p+q-1)}.
\end{align*}
This matrix has the following properties, which can be  checked by matrix multiplication.
\begin{align}
i_{1,p+q-1}&=i_{p,q}U_{p,q} \label{U1}\\
U_{p,q}S_{1,p+q-1}&=S_{p,q}U_{p,q} \label{U2}\\
U_{p,q}D_{1,p+q-1}&=D_{p,q}(U_{p,q} \otimes U_{p,q}) \label{U3}\\
U_{p,q}e_{1,p+q-1}&=e_{p,q} \label{U4}
\end{align}

\begin{claim}
The $(1,p+q-1)$ ornamented $T$ is cospectral with the $(p,q)$ ornamented $T$.
\end{claim}

Since being cospectral is an equivalence relationship, once this claim is established, then the proof of Theorem \ref{thm:Agraphs} is done. We quickly demonstrate the idea with an example using the word for the tree in Figure~\ref{fig:ornamented}, namely $iSSD((Se){\otimes}(e))$.  Recall from the previous claim that to compute the characteristic polynomial, we translate the word into a product of matrices. Using the various properties of $U_{p,q}$ given above together with basics of tensor products (e.g.\ $(A\otimes B)(C\otimes D)=(AC)\otimes (BD)$) we have
\begin{align*}
&\hspace{-40pt}i_{1,p+q-1}S_{1,p+q-1}S_{1,p+q-1}D_{1,p+q-1}((S_{1,p+q-1}e_{1,p+q-1})\otimes(e_{1,p+q-1}))\\
&=i_{p,q}U_{p,q}S_{1,p+q-1}S_{1,p+q-1}D_{1,p+q-1}((S_{1,p+q-1}e_{1,p+q-1})\otimes(e_{1,p+q-1}))\\
&=i_{p,q}S_{p,q}U_{p,q}S_{1,p+q-1}D_{1,p+q-1}((S_{1,p+q-1}e_{1,p+q-1})\otimes(e_{1,p+q-1}))\\
&=i_{p,q}S_{p,q}S_{p,q}U_{p,q}D_{1,p+q-1}((S_{1,p+q-1}e_{1,p+q-1})\otimes(e_{1,p+q-1}))\\
&=i_{p,q}S_{p,q}S_{p,q}D_{p,q}(U_{p,q}\otimes U_{p,q})((S_{1,p+q-1}e_{1,p+q-1})\otimes(e_{1,p+q-1}))\\
&=i_{p,q}S_{p,q}S_{p,q}D_{p,q}((U_{p,q}S_{1,p+q-1}e_{1,p+q-1})\otimes(U_{p,q}e_{1,p+q-1}))\\
&=i_{p,q}S_{p,q}S_{p,q}D_{p,q}((S_{p,q}U_{p,q}e_{1,p+q-1})\otimes(e_{p,q}))\\
&=i_{p,q}S_{p,q}S_{p,q}D_{p,q}((S_{p,q}e_{p,q})\otimes(e_{p,q})).
\end{align*}

The general case is again done by induction on the length of the word working from the end. It works for the word $e$ by \eqref{U4}. If it works for $W$, then for $SW$ along with \eqref{U2} we have that 
\[
U_{p,q}S_{1,p+q-1}W_{1,p+q-1}=S_{p,q}U_{p,q}W_{1,p+q-1}
\]
and the induction hypothesis finishes the case. If it works for the words $W$ and $W'$, then for $D((W){\otimes}(W'))$ along with \eqref{U3} we have that
\begin{align*}
U_{p,q}D_{1,p+q-1}((W_{1,p+q-1})\otimes(W_{1,p+q-1}'))&=
D_{p,q}(U_{p,q}\otimes U_{p,q})((W_{1,p+q-1})\otimes(W_{1,p+q-1}'))\\&=
D_{p,q}((U_{p,q}W_{1,p+q-1})\otimes(U_{p,q}W_{1,p+q-1}'))
\end{align*}
and again the induction hypothesis finishes the case. Finally, the $i_{p,q}$ term at the start of the extended word helps to introduce the $U_{p,q}$ term to get the process started for any word by \eqref{U1}. So the claim, and hence the theorem is verified.
\end{proof}

\subsection{Comments about $k$-ary trees}

Instead of considering only binary trees one could consider ornamenting $k$-ary trees. The first non-binary case of this is shown in Figure~\ref{fig:tri} which has cospectral graphs. The general approach would be similar to what we have outlined here (though now there would be multiple different matrices associated with the splitting). However, even if this was carried out, we would still only have families on vertices of the form $n=1+\alpha k$ for $\alpha\ge3$.

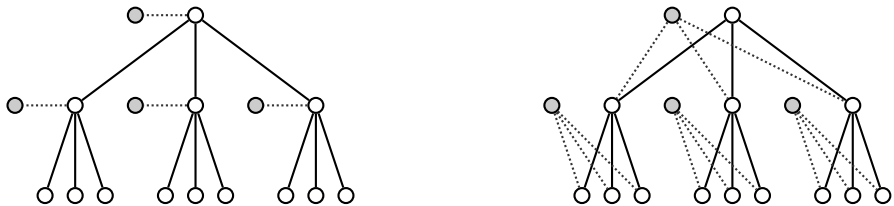
\begin{figure}[!htb]
\centering
\begin{tikzpicture}[scale=0.8]
\tikzstyle{vertex}=[inner sep=2pt,thick,shape=circle,draw=black];
\tikzstyle{ornament}=[inner sep=2pt,thick,shape=circle,draw=black,fill=black!20!white];
    \node[vertex] (A) at (0,0) {};
    \node[vertex] (B) at (-2,-1.5) {};
    \node[vertex] (C) at (0,-1.5) {};
    \node[vertex] (D) at (2,-1.5) {};
    \node[vertex] (B1) at (-2.5,-3) {};
    \node[vertex] (B2) at (-2,-3) {};
    \node[vertex] (B3) at (-1.5,-3) {};
    \node[vertex] (C1) at (-0.5,-3) {};
    \node[vertex] (C2) at (0,-3) {};
    \node[vertex] (C3) at (0.5,-3) {};
    \node[vertex] (D1) at (1.5,-3) {};
    \node[vertex] (D2) at (2,-3) {};
    \node[vertex] (D3) at (2.5,-3) {};
    \node[ornament] (AA) at (-1,0) {};
    \node[ornament] (BB) at (-3,-1.5) {};
    \node[ornament] (CC) at (-1,-1.5) {};
    \node[ornament] (DD) at (1,-1.5) {};
    
    \draw[thick] (A)--(B) (A)--(C) (A)--(D) (B)--(B1) (B)--(B2) (B)--(B3) (C)--(C1) (C)--(C2) (C)--(C3) (D)--(D1) (D)--(D2) (D)--(D3);
    
    \draw[ thick,color=black!80!white, densely dotted] (A)--(AA) (B)--(BB) (C)--(CC) (D)--(DD);

\end{tikzpicture}
\hfil
\begin{tikzpicture}[scale=0.8]
\tikzstyle{vertex}=[inner sep=2pt,thick,shape=circle,draw=black];
\tikzstyle{ornament}=[inner sep=2pt,thick,shape=circle,draw=black,fill=black!20!white];
    \node[vertex] (A) at (0,0) {};
    \node[vertex] (B) at (-2,-1.5) {};
    \node[vertex] (C) at (0,-1.5) {};
    \node[vertex] (D) at (2,-1.5) {};
    \node[vertex] (B1) at (-2.5,-3) {};
    \node[vertex] (B2) at (-2,-3) {};
    \node[vertex] (B3) at (-1.5,-3) {};
    \node[vertex] (C1) at (-0.5,-3) {};
    \node[vertex] (C2) at (0,-3) {};
    \node[vertex] (C3) at (0.5,-3) {};
    \node[vertex] (D1) at (1.5,-3) {};
    \node[vertex] (D2) at (2,-3) {};
    \node[vertex] (D3) at (2.5,-3) {};
    \node[ornament] (AA) at (-1,0) {};
    \node[ornament] (BB) at (-3,-1.5) {};
    \node[ornament] (CC) at (-1,-1.5) {};
    \node[ornament] (DD) at (1,-1.5) {};
    
    \draw[thick] (A)--(B) (A)--(C) (A)--(D) (B)--(B1) (B)--(B2) (B)--(B3) (C)--(C1) (C)--(C2) (C)--(C3) (D)--(D1) (D)--(D2) (D)--(D3);
    
    \draw[ thick,color=black!80!white, densely dotted] (AA)--(B) (AA)--(C) (AA)--(D) (BB)--(B1) (BB)--(B2) (BB)--(B3) (CC)--(C1) (CC)--(C2) (CC)--(C3) (DD)--(D1) (DD)--(D2) (DD)--(D3);

\end{tikzpicture}

\caption{The first ornamented $k$-ary pair that does not come from some binary tree; the two graphs are cospectral with respect to $\mathcal{A}$.}
\label{fig:tri}
\end{figure}

So, constructions of the type discussed in this section will be least informative for the case when $n=1+p$ where $p$ is a prime as there are no ornamented binary trees available (beyond complete bipartite graphs). Reconsidering Table~\ref{tab:AA}, it is of interest to note that for the line looking at the counts of spectral faux trees for $\mathcal{A}$, that the big drops seem to occur for such values of $n$ (e.g.\ $n=8$, $12$, $14$).

\section{Conclusion}\label{sec:conclusion}
We have gone from the question of Schwenk ``Can you hear the shape of a tree?'' to ``Can you hear whether a graph is a tree?'' For the adjacency, the commonality of spectral faux trees gives an answer of ``almost never''. For the Laplacian the answer is ``yes'' while for the signless Laplacian the answer is ``frequently'' (and probably almost always).

For the normalized adjacency our current best answer is ``maybe''. The issue is that while there do exist spectral faux trees for small values of $n$, there do not appear to be many of them (and the number of spectral faux trees also seems to be quite sporadic as seen in Table~\ref{tab:AA}). We have constructed a family that grows exponentially for the normalized adjacency. However, this is a highly specialized family and still leaves gaps for a large number of values of $n$. This leads to the following open question.

\begin{question}
Do almost all trees have a cospectral mate which is a non-tree for the normalized adjacency matrix?
\end{question}

If the answer to this question is ``yes'', one might try to approach it by finding a small pair of graphs, one a tree and one a non-tree, which can have an arbitrary graph glued onto a single vertex (as we did with the adjacency) or a group of vertices (as we did with the signless Laplacian). However, up through $n=14$ vertices, no such pair exists. (If such a pair had existed, the construction of an exponentially large cospectral family would have been much simpler.) Much work remains to be done for understanding the normalized adjacency matrix.

\subsection*{Acknowledgments}
This research was conducted primarily at the 2022 Iowa State University Math REU which was supported through NSF Grant DMS-1950583.

\bibliographystyle{plain}
\bibliography{bibliography}

\begin{thebibliography}{10}

\bibitem{haemers}
Andries~E. Brouwer and Willem~H. Haemers.
\newblock {\em Spectra of graphs}.
\newblock Universitext. Springer, New York, 2012.

\bibitem{brualdi}
Richard~A. Brualdi and Herbert~J. Ryser.
\newblock {\em Combinatorial matrix theory}, volume~39 of {\em Encyclopedia of
  Mathematics and its Applications}.
\newblock Cambridge University Press, Cambridge, 1991.

\bibitem{butler}
Steve Butler.
\newblock A jaunt in spectral graph theory.
\newblock In {\em 50 years of combinatorics, graph theory, and computing},
  Discrete Math. Appl. (Boca Raton), pages 213--237. CRC Press, Boca Raton, FL,
  [2020] \copyright 2020.

\bibitem{coalescing}
Steve Butler, Elena D'{A}vanzo, Rachel Heikkinen, Joel Jeffries, Alyssa
  Kruczek, and Harper Niergarth.
\newblock Complements of coalescing sets.
\newblock Preprint.

\bibitem{heysse}
Steve Butler and Kristin Heysse.
\newblock A cospectral family of graphs for the normalized {L}aplacian found by
  toggling.
\newblock {\em Linear Algebra Appl.}, 507:499--512, 2016.

\bibitem{signless_matrix_tree}
Keivan Hassani~Monfared and Sudipta Mallik.
\newblock An analog of matrix tree theorem for signless {L}aplacians.
\newblock {\em Linear Algebra Appl.}, 560:43--55, 2019.

\bibitem{kac}
Mark Kac.
\newblock Can one hear the shape of a drum?
\newblock {\em Amer. Math. Monthly}, 73(4, part II):1--23, 1966.

\bibitem{schwenk}
Allen~J. Schwenk.
\newblock Almost all trees are cospectral.
\newblock In {\em New directions in the theory of graphs ({P}roc. {T}hird {A}nn
  {A}rbor {C}onf., {U}niv. {M}ichigan, {A}nn {A}rbor, {M}ich., 1971)}, pages
  275--307. Academic Press, New York, 1973.

\bibitem{stanley}
Richard~P. Stanley.
\newblock {\em Catalan numbers}.
\newblock Cambridge University Press, New York, 2015.

\bibitem{vandam}
Edwin~R. van Dam and Willem~H. Haemers.
\newblock Which graphs are determined by their spectrum?
\newblock volume 373, pages 241--272. 2003.
\newblock Special issue on the Combinatorial Matrix Theory Conference (Pohang,
  2002).

\end{thebibliography}

\end{document}